\documentclass[reqno,12pt]{amsart}

\usepackage{amsmath}
\usepackage{amsfonts}
\usepackage{amsthm}
\usepackage{amssymb}
\usepackage{array}
\usepackage{booktabs}
\usepackage{color}
\usepackage{changepage}
\usepackage[english]{babel}
\usepackage{enumerate}
\usepackage{epsfig}
\usepackage{float}
\usepackage{framed}
\usepackage{gensymb}
\usepackage{graphicx}
\usepackage{hyperref}
\usepackage{indentfirst}
\usepackage{longtable}
\usepackage{mathtools}
\usepackage{multirow}
\usepackage[normalem]{ulem}
\usepackage{setspace}
\usepackage{subcaption}
\usepackage{textcomp}
\usepackage[utf8]{inputenc}
\usepackage{verbatim} 
\usepackage{enumerate}
\usepackage{xlop}
\usepackage{listings}
\usepackage[top=1 in,bottom=1in, left=1 in, right=1 in]{geometry}

\newtheorem{theorem}{Theorem}[section]

\newtheorem{proposition}[theorem]{Proposition}

\newtheorem{conjecture}[theorem]{Conjecture}
\numberwithin{equation}{section}

\def\and{%
  \end{tabular}%
  \hskip 1em \@plus.17fil%
  \begin{tabular}[t]{c}}%
  
\begin{document}

\title{On Banzhaf and Shapley-Shubik Fixed Points and Divisor Voting Systems}

\author[Arnell, Chen, Choi, Marinov, Polina, Prakash]{Alex Arnell, Richard Chen, Evelyn Choi, Miroslav Marinov, Nastia Polina, Aaryan Prakash \\ Research Mentor: Joshua Zelinsky }

\date{October 16, 2020}

\maketitle

\noindent

\setcounter{section}{0}

\begin{abstract}
    The Banzhaf and Shapley-Shubik power indices were first introduced to measure the power of voters in a weighted voting system. Given a weighted voting system, the fixed point of such a system is found by continually reassigning each voter's weight with its power index until the system can no longer be changed by the operation. We characterize all fixed points under the Shapley-Shubik power index of the form $(a,b,\ldots,b)$ and give an algebraic equation which can verify in principle whether a point of this form is fixed for Banzhaf; we also generate Shapley-Shubik fixed classes of the form $(a,a,b,\ldots,b)$.  We also investigate the indices of divisor voting systems of abundant numbers and prove that the Banzhaf and Shapley-Shubik indices differ for some cases.  
    
\end{abstract}

\section{Introduction}
In a weighted voting system, voters, or \textit{players}, have different amounts of the total votes, which are called \textit{weights}. A \textit{motion} is an agenda item that needs some amount of votes to be passed. This amount is called the \textit{quota}. A \textit{coalition} is a group of players, and the sum of all votes in a coalition is called the total voting power. A coalition is a \textit{winning coalition} if its total voting power meets or exceeds the quota. A player is called a \textit{critical player} in a winning coalition if the coalition would not be winning without the player. 

We examine the \textit{Banzhaf power index} \cite{Banzhaf} and the \textit{Shapley-Shubik power index} \cite{Shapley-Shubik}, which are two different methods of measuring a player's strength in a system. The Banzhaf power index of a player is the number of times that player is a critical player in all winning coalitions divided by the number of total times any player is a critical player.
The Shapley-Shubik index looks at permutations of all players in a system, called \textit{sequential coalitions}. We sum each player's votes starting from the beginning of a sequential coalition, and see if the sum reaches the quota as we progress. The player whose votes first cause this sum to meet or exceed the quota is called a \textit{pivotal player}. The Shapley-Shubik power index of a player is the number of times that player is a pivotal player divided by the total number sequential coalitions.

We present results in two different directions. The first is slightly more number-theoretic, while the second relies heavily on modelling with suitable equivalent algebraic equations. 
\medskip

For an integer $n> 1$, the \textit{divisor voting system of $n$} is $[Q: n, d_k, \ldots, d_1, 1]$ where $d_1,\ldots,d_k$ are the divisors of $n$ distinct from $1$ and $n$; also $Q = (\sigma(n)+1)/2$ when $\sigma(n)$ is even and $Q= \sigma(n)/2$ when $\sigma(n)$ is odd -- so the quota is essentially a simple majority rule. Here $\sigma(n)$ denotes the sum of the divisors of $n$.

In Section 2 we will show that for any positive integer $n$ where $\sigma(n)\geq 2n$, the Banzhaf power index and Shapley-Shubik index disagree on at least one
divisor for the divisor voting system of $n$ for all $n$ such that $\sigma(n) = 2n+k$, $0\leq k \leq 5$. It is unlikely that such an approach would work in the general case and so it is more plausible to examine an inductive argument of the following sort -- given $n$, what can we say about $pn$ and $mn$, where $p$, $m$ are primes? We have some progress towards this approach in Section $2$.

\medskip 

Now let us introduce the algebraic direction of our research in Section 3.

Consider systems of the form $V = [1/2_s : a_1,\ldots, a_n]$ where $a_1,\ldots,a_n \geq 0$ are real numbers that sum to $1$ and all winning coalitions are those whose sum \textit{strictly exceeds} $1/2$. We can evaluate the Banzhaf and Shapley-Shubik powers of each player, thus obtaining tuples $BB(V)$ and $SS(V)$. We say that $V$ is a \textit{Banzhaf fixed point} if $V = BB(V)$ and a \textit{Shapley-Shubik fixed point} if $V = SS(V)$. A fixed point is \textit{primitive} if there is no player $i$ with $a_i = 0$. A fixed point is \textit{non-trivial} if not all players have the same power. 
    
In Section $3$ we derive algebraic equations whose solutions generate fixed points and in particular in Section $3.2$ we give new classes of Shapley-Shubik fixed points. We then categorize all fixed points under Shapley-Shubik and use this to show that there is only one fixed point for both Banzhaf and Shapley-Shubik. Note that in general the problem is difficult to model with an equation and so we will restrict to the (already far from trivial) tuples of one of the forms $(a,b,\ldots,b)$ and $(a,a,b,\ldots,b)$.

\section{Divisor Voting Systems and Abundant Numbers}
Our first main aim in this section is to prove the following.

\begin{proposition}
The Banzhaf power index and Shapley-Shubik index are different for at least one divisor for the divisor voting system of $n$ when $\sigma(n) = 2n+k$, $0\leq k \leq 5$.
\end{proposition}

An integer $n$ is called \textit{abundant} if $\sigma(n) > 2n$. We shall sometimes refer to the following auxilliary observation.

\begin{proposition}
\label{abundantdivisors}
Any abundant number must have more than 5 divisors, and the only abundant number with 6 divisors is 20.
\end{proposition}

\begin{proof}
Let $g: \mathbb{Z}\rightarrow\mathbb{R}$ be the function defined by $g(n) = \frac{n}{\phi{(n)}}$. By substituting Euler's product formula for $\phi{(n)}$, we have $g(n)={\displaystyle \prod_{p \mid n} \frac{p}{p-1}}$. 

Suppose $n$ is an odd abundant number. Then, since $n$ is abundant, $\frac{\sigma{(n)}}{n} > 2$. Since $g(n) > \frac{\sigma{(n)}}{n}$ when $n>1$, $g(n) > 2$. Now we will find the minimum number of divisors $n$ can have. The function $f(x)=\frac{x}{x-1}$ is a decreasing function over the interval $[2, \infty)$. Therefore, the two prime odd numbers for which $\frac{p}{p-1}$ is at its largest values are $3$ and $5$. However, $(\frac{3}{2})(\frac{5}{4}) < 2$. Therefore, $n$ must have at least 3 prime odd divisors to be abundant, or 8 total divisors.

Suppose $n$ is an even abundant number. Since $n$ is abundant, $g(n) >2$. In this case, it is possible for $n$ to have only 2 prime divisors. These prime divisors must be 2 and 3, since $(\frac{2}{1})(\frac{3}{2}) > 2$. 2 and 3 are the only prime divisors for which this is possible, since $f(x) = \frac{x}{x-1}$ is decreasing for $x\geq 2$, and $(\frac{3}{2})(\frac{5}{4}) < 2$. 20 is the only abundant number with only 2 and 3 as prime divisors. Hence, 20 is the only abundant number with 6 divisors, and all other even abundant numbers must have at least 8 divisors.
\end{proof}

\subsection{Casework -- Proof of Proposition $2.1$} 
In this part $d$ is the number of divisors of $n$. We will prove Proposition $2.1$ by breaking it into cases based on the value of $\sigma{(n)}$.

\subsubsection{$\sigma(n)=2n$}
\label{subsubsection:sigma(n)=2n}
\begin{proof}

We will first find the Banzhaf power index of each divisor of $n$ based on the number of divisors $d$. We can calculate the number of times $n$ is a critical player in the divisor voting system of $n$. Consider the total number of possible combinations of the divisors that include $n$, $2^{d-1}$. Since the sum of all divisors excluding $n$ is also $n$, and the quota is $n+1$, the only combination that is not a winning coalition is the coalition that contains only $n$. Hence $n$ is a critical player $2^{d-1} -1$ times.

Every divisor that is not $n$ is a critical player only when it is in the winning coalition consisting of $n$ and itself. So, the Banzhaf index for $n$ is $\frac{2^{d-1}-1}{2^{d-1}+(d-2)}$. The Banzhaf index for any other divisor $d_i$ of $n$ is $\frac{1}{2^{d-1}+(d-2)}$.

We will now calculate the Shapley-Shubik power index for each divisor of $n$. A proper divisor $d_i$ of $n$ can be a pivotal player only once. It is a pivotal player in the sequential coalition where $n$ is the first player followed by $d_i$. Hence, the number of times $d_i$ is a pivotal player is $(d-2)!$. We can now find the number of times $n$ is a pivotal player. There are $d-1$ proper divisors of $n$ and $d!$ total pivotal players, so the number of times $n$ is a pivotal player is $d!-(d-2)!(d-1)$. So $n$ has a Shapley-Shubik index of $\frac{d!-(d-1)!}{d!}$ and every other divisor of $n$ has index  $\frac{(d-2)!}{d!} = \frac{1}{d(d-1)}$. 

We now compare the Shapley-Shubik indices and the Banzhaf indices to show that they differ for at least one divisor of $n$. We can show that each proper divisor of $n$, $d_i$, has a Banzhaf index that differs from its Shapley-Shubik index by setting equal the their formulas we derived. So, we have 
\begin{equation*}
    \frac{1}{2^{d-1}+(d-2)}=\frac{1}{d(d-1)}.
\end{equation*}
The only integer solution to this equation is $d=2$ -- indeed, note that $2^{d-1} > d^2 - 2d + 2$ for $d\geq 6$ by induction, with the inductive step being $2^d = 2 \cdot 2^{d-1} > 2(d^2 - 2d + 2) > d^2 + 2 = (d+1)^2 - 2(d+1) + 2$ (the second inequality is equivalent to $d(d-4) + 2 > 0$). However, since $n$ is an abundant number, it must have more than six divisors by Proposition \ref{abundantdivisors}.
Thus, for every perfect number, at least one of its divisors in its divisor voting system will have a Banzhaf index that differs from its Shapley-Shubik index.
\end{proof}

We followed a similar procedure to find and compare the Banzhaf and Shapley-Shubik power indices for each divisor in the divisor voting system of $n$ where $\sigma(n)=2n+k$. Note that we will mostly omit the derivations of the calculations, as well as reasons why two concrete expressions involving $d$ are equal apart from a few cases, since the ideas are largely the same.

\subsubsection{$\sigma(n)=2n+1$}
Since both the Banzhaf and Shapley-Shubik power indices of 1 are 0, we must compare the Banzhaf and Shapley-Shubik power index formulas for proper divisors $d_i$ that are not 1. Using the same method that used in 2.1.1, we can see that the formula for the Banzhaf index of each $d_i$ is $\frac{2}{2^{d-1}+2(d-2)}$. The formula for the Shapley-Shubik index of each $d_i$ is $\frac{2(d-2)!}{d!} = \frac{2}{d(d-1)}$. These two formulas are equal only when the number of divisors of $n$ is 2, 3, or 4, which is not possible since $n$ is an abundant number. Note that numbers $n$ where $\sigma(n)=2n+1$, are called \textit{quasiperfect numbers}, and it is unknown if any such numbers really exist \cite{quasiperfect}.

\subsubsection{$\sigma(n)=2n+2$}
Both the Banzhaf and Shapley-Shubik power indices of 1 do not equal 0. Therefore we can compare the Banzhaf and Shapley-Shubik indices for 1 to prove that 2.0.1 is true for this case. The formula for the Banzhaf index of 1 is $\frac{1}{2^{d-1}+3(d-2)-2}$ and the formula for the Shapley-Shubik index of 1 is $\frac{(d-2)!}{d!} = \frac{1}{d(d-1)}$. If these two expressions are equal to each other, then $d$ must be 4, but an abundant number must have at least 6 divisors. Hence when $\sigma(n)=2n+2$, the power indices of 1 in Shapley-Shubik and Banzhaf differ from one another.

\subsubsection{$\sigma(n)=2n+3$}
In this case, since both the Banzhaf and Shapley-Shubik power indices of 1 are 0, we must compare the power indices of divisors $d_i$ that are not 1, nor $n$. Then the Banzhaf index of each $d_i$ is $\frac{4}{2^{d-1}+4(d-2)-4}$ and the Shapley-Shubik index of $d_i$ is $\frac{4(d-3)!+2(d-2)!}{d!} = \frac{2}{(d-1)(d-2)}$. There are no integer solutions for when the two indices of $d_i$ are equal -- indeed, we can prove that $2^{d-1} > 2d^2 - 10d+16$ for $d\geq 6$ with induction. Suppose that the inequality is true for $d$. By manipulating this inequality, we have that $2^d = 2\cdot 2^{d-1} > 2(2d^2 - 10d+16) > 4d^2 - 20d - 32 = 2(d+1)^2 - 10(d+1)+16$ (the second inequality is equivalent to $2(d-4)(d-3) > 0$). Therefore, since the inequality is true for $d+1$, it is true for all $d \geq 6$. Hence, since an abundant number must have at least 6 divisors, there are no solutions for when the two indices are equal.

\subsubsection{$\sigma(n)=2n+4$}
We must consider two cases for when $n$ is even or odd. Suppose $n$ is even. Since the index of 1 is not 0, we can simply compare the Banzhaf and Shapley-Shubik indices for 1. The Banzhaf index of 1 is $\frac{1}{2^{d-1}+5(d-3)-1}$. The Shapley-Shubik index of 1 is $\frac{2(d-3)!}{d!} = \frac{2}{d(d-1)(d-2)}$. When we set these two expressions for the Banzhaf and Shapley-Shubik index equal and simplify, we have $d^3+2d-3d^2 = 2^d + 10d-32$. We will show that these expressions differ since $n$ is an abundant number and must have more than 5 divisors.
Let $h: \mathbb{R}\rightarrow\mathbb{R}$ be the function defined by $h(x) = 2^x + 10x-32-(x^3+2x-3x^2)$. Thus $h(x) = 2^x-x^3 + 3x^2 + 8x-32$. By Proposition \ref{abundantdivisors}, an abundant number must have at least 6 divisors. The integers 6, 7, and 8 generate values of $h(x)$ that are not equal to 0. To show that $h(x) > 0$ for values $x\geq 9$, it is sufficient to show that $h(9) > 0$ and $h'(x) = 0$, which is true if $h^{(3)}(9) > 0$ and $h^{(4)}(x)>0$ for all $x\geq 9$.
$h^{(3)}(x) = 2^xln^3(2)-6$, which is greater than 0 for all numbers greater than or equal to 9, and $h^{(4)}(x) = 2^xln^4(2)$, which is clearly greater than 0 for all values of $x$.
Hence the two expressions for the Banzhaf and Shapley-Shubik index of 1 are never equal for an abundant number $n$.

Now suppose $n$ is odd. Since the index of 1 is 0 for both Banzhaf and Shapley-Shubik indices, we must compare the formulas for the indices of proper divisors $d_i$ that are not 1. The formula for the Banzhaf index of each $d_i$ is $\frac{4}{2^{d-1}+4(d-2)-3}$, and the formula for the Shapley-Shubik index of $d_i$ is $\frac{2[(d-2)\cdot (d-2)!]+4(d-2)!}{d!} = \frac{2}{d-1}$. By simplifying $\frac{4}{2^{d-1}+4(d-2)-3} = \frac{2}{d-1}$ we have $2^{d-1}=-2d+9$. However, $2^{d-1} > -2d+9$ for all $d > 5$, since the graph $y=-2x+9$ intersects the x-axis only when $x<5$, and $y=2^{d-1}$ is an increasing function when $x\geq 0$.
Therefore, since $n$ is an abundant number and must have at least $6$ divisors, there are no abundant numbers for which these expressions are equal.

\subsubsection{$\sigma(n)=2n+5$}
We again have two cases to consider for when $n$ is even or odd. Suppose $n$ is even. The Banzhaf and Shapley-Shubik indices of 1 are not 0, so we can simply compare these two indices for 1. The formula for the Banzhaf index of 1 is $\frac{1}{2^{d-1}+5(d-3)-2}$ and the formula for the Shapley-Shubik index of 1 is $\frac{2(d-3)!}{d!} = \frac{2}{d(d-1)(d-2)}$. Equating the expressions for the Banzhaf and Shapley-Shubik indices gives $2^d+10d-34 = d^3+2d-3d^2$. $n$ is an abundant number and thus must have at least 6 divisors by Proposition 2.2. Hence we will show that this equation has no integers solutions greater than 5 using the same method as in 2.1.5. Let $h: \mathbb{R}\rightarrow\mathbb{R}$ be the function defined by $h(x) = 2^x+10x-34 - (x^3+2x-3x^2)$. Hence $h(x) = 2^x-x^3+3x^2+8x-34$. We can manually verify that $h(6)$, $h(7)$, and $h(8)$ are not equal to 0. We will now show that if $x$ is greater than or equal to 9, $h(x)$ is greater than 0. We can demonstrate this by showing that $h^{(3)}(9)>0$ and $h^{(4)}(x) > 0$ for all $x\geq9$. $h^{(3)}(9) = 2^9\ln{(2)}^{3} - 6$, which is greater than 0 and $h^{(4)}(x) = 2^x\ln{(2)}^4$, which is also clearly greater than 0. 
Hence there are no integer solutions for when the Banzhaf index of 1 and the Shapley-Shubik index of 1 are equal.

In the second case, $n$ is odd. Since the Banzhaf and Shapley-Shubik indices of 1 are both 0, we must compare the formulas for the indices of proper divisors $d_i$ where $d_i$ is not 1. The Banzhaf index of each $d_i$ is $\frac{4}{2^{d-1}+4(d-2)-6}$ and the formula for the Shapley-Shubik index of $d_i$ is $\frac{2[(d-2)\cdot(d-2)!]+4(d-2)!}{d!} = \frac{2}{d-1}$. By simplifying $\frac{4}{2^{d-1}+4(d-2)-6} = \frac{2}{d-1}$, we have $2^{d-1}=-2d+12$. Since $f(x)=2^{x-1}$ is an increasing function for $x\geq 0$ and $g(x)=-2x+12$ is a decreasing function that intersects the x-axis when $x=6$, $2^{d-1}>-2d+12$ for all $x\geq 6$.
Hence there are no possible abundant numbers $n$ for when these two expressions are equal. 

\subsection{Forms $pn$ and $mn$}
In order to extend the class of abundant numbers for which Proposition 2.1 is true, we can examine numbers of the form $pn$ and $mn$ where $p$ and $m$ are primes. 

\begin{conjecture}
Let $n$ be an abundant number. If $p$ and $m$ are primes greater than $\sigma{(n)}$, then the divisors of the divisor voting systems of $pn$ and $mn$ have the same Shapley-Shubik power indices as well as the same Banzhaf power indices.
\end{conjecture}

We will prove a part of this conjecture, which is the following proposition.

\begin{proposition}
Let $n$ be an abundant number and let $p$ and $m$ be primes greater than $\sigma{(n)}+1$. Then the divisor voting systems of $pn$ and $mn$ have the same number of winning coalitions.

\end{proposition}
\begin{proof} 
Let the divisor voting system of $n$ be $\{1, d_1, d_2,\dots, d_k, n\}$, where each $d_i$ is a proper divisor of $n$. Since $p$ and $m$ are primes, the divisor voting systems of $pn$ and $mn$, $S_{pn}$ and $S_{mn}$ respectively, are the following:

\begin{center}
    $S_{pn}=\{1, d_1,\dots, d_k, n, p, pd_1, pd_2, \dots, pd_k, pn\}$
\end{center}
\begin{center}
    $S_{mn}=\{1, d_1,\dots, d_k, n, m, md_1, md_2, \dots, md_k, mn\}$.
\end{center}

Let $W_{pn}$ be the set of winning coalitions in $S_{pn}$ and let $W_{mn}$ be the set of all winning coalitions in $S_{mn}$. In order to prove that $W_{pn}$ and $W_{mn}$ have the same number of coalitions, we will prove that there exists a bijection between the two sets.

Let $f\colon W_{mn} \rightarrow W_{pn}$ be the function defined by 
\begin{equation*}
f(\{d_{j_1},d_{j_2},\dots, d_{j_l}, md_{c_1},md_{c_2},\dots, md_{c_s}\}) = \{d_{j_1},d_{j_2},\dots, d_{j_l}, pd_{c_1},pd_{c_2},\dots, pd_{c_s}\}. 
\end{equation*}

Note that $f$ is injective, since if $$\{d_{j_1},d_{j_2},\dots, d_{j_l}, pd_{c_1},pd_{c_2},\dots, pd_{c_s}\} = \{d_{j_1}',d_{j_2}',\dots, d_{j_l}', pd_{c_1}',pd_{c_2}',\dots, pd_{c_s}'\},$$ then we have  $$\{d_{j_1},d_{j_2},\dots, d_{j_l}, md_{c_1},md_{c_2},\dots, md_{c_s}\} = \{d_{j_1}',d_{j_2}',\dots, d_{j_l}', md_{c_1}',md_{c_2}',\dots, md_{c_s}'\}.$$

Now we prove that $f$ is surjective. Suppose we have a winning coalition $G$ in $S_{pn}$. Then we will prove that there exists a winning coalition $H$ in $S_{mn}$ such that $f$ maps $H$ to $G$. 

$G$ is of the form
$[d_j, d_{j+1},\dots, d_{j+k}, pd_c,pd_{c+1},\dots, pd_{c+r}]$, where each $d_i$ is a divisor of $n$. Then, based on our construction of $f$, we know that $H$ should be of the form

\begin{center}
$[d_j, d_{j+1},\dots, d_{j+k}, md_c,md_{c+1},\dots, md_{c+r}]$.
\end{center}

We will now prove that $H$ is a winning coalition.

Let $q_{pn}$ be the quota of $S_{pn}$. Then 
\begin{equation*}
q_{pn}=\dfrac{p\sigma{(n)}+\sigma{(n)}}{2}+1.
\end{equation*}

Similarly, for the quota of $S_{mn}$ we have
\begin{equation*}
q_{mn}=\dfrac{m\sigma{(n)}+\sigma{(n)}}{2}+1.
\end{equation*}

Let $w_1$ be the total voting power of $G$. Since $G$ is a winning coalition, we have 
\begin{equation}
w_1 \geq q_{pn} = \dfrac{p\sigma{(n)}+\sigma{(n)}}{2}+1.
\end{equation}

Let $$\beta=\sum_{n=1}^{k} d_{j+n}$$ and let $$\alpha =\sum_{n=1}^{r} d_{c+n}.$$
Then we can rewrite $w_1$ as $p\beta+\alpha$. By substituting $p\beta+\alpha$ for $w_1$ we have
\begin{equation}
    p\beta+\alpha \geq \dfrac{p\sigma{(n)}+\sigma{(n)}}{2}+1.
\end{equation}
Hence
\begin{equation}
p \geq\dfrac{\sigma({n})+2-2\alpha}{2\beta-\sigma({n})}.
\end{equation}

The maximum value of the right-hand side of 2.3 occurs when the denominator is 1 and $\alpha$ is 0. Hence the right-hand side of 2.3 is at most $\sigma(n)+2$. Since $m > \sigma(n) +1$, we have
\begin{equation}
m \geq\dfrac{\sigma({n})+2-2\alpha}{2\beta-\sigma({n})}
\end{equation}

By rearranging 2.4, we have
\begin{equation*}
    m\beta+\alpha \geq \dfrac{m\sigma{(n)}+\sigma{(n)}}{2}+1=q_{mn}.
\end{equation*}
The total voting power of $H$ is also $m\beta + \alpha$ due to its construction. Hence $H$ is a winning coalition and $f$ is surjective. 
Therefore, since there exists a bijection between $W_{mn}$ and $W_{pn}$, the divisor voting systems of $mn$ and $pn$ have the same number of winning coalitions.
\qedhere 
\end{proof}

\section{Fixed Points}

\subsection{Banzhaf}
The general Banzhaf power indices for an arbitrary distribution of weights is quite difficult to express, so we will only consider the weights of the form $(a, b, b, \dots , b)$ where there is one player with a vote of $a$ and $m$ players have a vote of $b$, and the weights sum to $1$. We refer to people with a vote of $a$ as type $A$ and people with a vote of $b$ as Type $B$. Then we have $a = 1 - mb$. Each winning coalition either contains the player of Type $A$, or it does not.

\begin{itemize}
    \item If a winning coalition does not contain the player of Type $A$, then there must be at least $p$ players of type $B$, where $pb > \frac{1}{2}$. Since $p$ is an integer, this simplifies to $p \geq \lceil \frac{1}{2b} \rceil$. Since all players in our coalition have the same voting power, we must have $p = \lceil \frac{1}{2b} \rceil$ for any player to be a critical player. Then each player of type $B$ is a critical player exactly $\binom{m-1}{\lceil \frac{1}{2b} \rceil - 1}$ times because we can pick $\lceil \frac{1}{2b} \rceil - 1$ of the remaining $m-1$ players of type $B$ to be the other players on our coalition.
    \item Now suppose a winning coalition contains the player of Type $A$. Then if are $q$ players of type $B$, $q$ must satisfy $1 - mb + qb > \frac{1}{2}$ and $qb < \frac{1}{2}$. $q$ is also an integer, so this simplifies to $m - \lfloor \frac{1}{2b} \rfloor \leq q \leq \lfloor \frac{1}{2b} \rfloor$. Then the player of Type $A$ is critical $\binom{m}{m - \lfloor \frac{1}{2b} \rfloor} + \binom{m}{m - \lfloor \frac{1}{2b} \rfloor + 1} + \dots + \binom{m}{\lfloor \frac{1}{2b} \rfloor}$ times, since we pick any $q$ players of type $B$. Furthermore, players of type $B$ can also be critical when $q$ is a minimum, so players of Type $B$ are critical an additional $\binom{m-1}{m - \lfloor \frac{1}{2b} \rfloor - 1}$ times.
\end{itemize}

Therefore the player of type A has Banzhaf index 
\[
 \frac{\binom{m}{m - \lfloor \frac{1}{2b} \rfloor} + \binom{m}{m - \lfloor \frac{1}{2b} \rfloor + 1} + \dots + \binom{m}{\lfloor \frac{1}{2b} \rfloor}}{m \left(\binom{m-1}{\lceil \frac{1}{2b} \rceil - 1} + \binom{m-1}{m - \lfloor \frac{1}{2b} \rfloor - 1} \right) + \left[\binom{m}{m - \lfloor \frac{1}{2b} \rfloor} + \binom{m}{m - \lfloor \frac{1}{2b} \rfloor + 1} + \dots + \binom{m}{\lfloor \frac{1}{2b} \rfloor} \right]}
\] 
while the player of type B has Banzhaf index
\[
 \frac{\binom{m-1}{\lceil \frac{1}{2b} \rceil - 1} + \binom{m-1}{m - \lfloor \frac{1}{2b} \rfloor - 1} }{m \left(\binom{m-1}{\lceil \frac{1}{2b} \rceil - 1} + \binom{m-1}{m - \lfloor \frac{1}{2b} \rfloor - 1} \right) + \left[\binom{m}{m - \lfloor \frac{1}{2b} \rfloor} + \binom{m}{m - \lfloor \frac{1}{2b} \rfloor + 1} + \dots + \binom{m}{\lfloor \frac{1}{2b} \rfloor} \right]}.
\]
We need the first expression to be equal to $1-mb$ and the second expression to be equal to $b$ (but note that it is sufficient to investigate one of these, since by symmetry if one of them is true then the other one becomes automatically true). It is difficult to find classes of solutions solely based on this equation, so we will investigate fixed points of both Banzhaf and Shapley-Shubik indices instead, in Section $3.3$. 

\subsection{Shapley-Shubik}
First let us examine the primitive non-trivial fixed points of $[1/2_s: a, b, b, \ldots, b]$, where the total sum of the votes is 1. We can now say that players of type $A$ have a weight of $a$ and players of type $B$ have a weight of $b$. When trying to compute the power of type $A$, suppose that there are $p$ players of type $B$ in the permutation before the person of type $A$. If the players of type $A$ is a critical player, then we know that $pb \leq \frac{1}{2}$ and $pb + a > \frac{1}{2}$. We know that $a = 1 - mb$, so we can substitute that into the second inequality to get $pb + (1 - mb) > \frac{1}{2}$. Solving these inequalities, we get $m - \frac{1}{2b} < p \leq \frac{1}{2b}$. We know that $p$ must be an integer, so assuming $\frac{1}{2b} \not \in \mathbb{Z}$, we can make the bounds stricter, which gives us $m - \lfloor \frac{1}{2b} \rfloor \leq p \leq \lfloor \frac{1}{2b} \rfloor$. To figure out how many different permutations correspond to each value of $p$, we can see that there must be $m!$ different permutations for how to arrange the players of type $B$, and then we just need to put the person of type $A$ in the appropriate location. Therefore, $m!(\lfloor \frac{1}{2b} \rfloor - (m - \lfloor \frac{1}{2b} \rfloor) + 1) = m!(2\lfloor \frac{1}{2b} \rfloor - m + 1)$ counts the number of permutations where type $A$ is the critical player, and $\frac{m!(2\lfloor \frac{1}{2b} \rfloor - m + 1)}{(m + 1)!} = \frac{2\lfloor \frac{1}{2b} \rfloor - m + 1}{m + 1}$ gives the power of the person of type $A$ if $\frac{1}{2b} \not \in \mathbb{Z}$. 

On the other hand, if $\frac{1}{2b} \in \mathbb{Z}$, then $m!\left(\frac{1}{2b} - (m - \frac{1}{2b})\right) = m!(\frac{1}{b} - m)$ represents the number of permutations such that the person of type $A$ is a critical player. This means that $\frac{1-mb}{b(m + 1)}$ is the power index of the person of type $A$. But note that this equals $1-mb$ if and only if $b= \frac{1}{m+1}$ in which case all players have the same power, i.e. we get a trivial fixed point!

\begin{proposition}Let $C$ be a positive integer. If $m = 2k - 1$, then for large enough $k$ $($depending only on $C)$, $b = \frac{k - C}{2k^2- k}$ and $a = 1 - (2k - 1)b$ yield a non-trivial fixed point of the form $[1/2_s: a, b, \dots, b]$. Moreover, all non-trivial fixed points are of this form.
\end{proposition}

\begin{proof}
If $b = \frac{k - C}{2k^2 - k}$, then $\frac{1}{2b} \neq \lfloor \frac{1}{2b} \rfloor$. We calculated the power of the person of type $A$ to be $\frac{2\lfloor \frac{1}{2b} \rfloor - m + 1}{m + 1}$, so $\frac{2\lfloor \frac{1}{2b} \rfloor - (2k - 1) + 1}{2k} = 1 - 2kb$ is necessary for this specific configuration to be a fixed point. If $b = \frac{k - C}{2k^2 - k}$, then $\left\lfloor \frac{1}{2 \cdot \frac{k - C}{2k^2 - k}} \right\rfloor = k + C - 1$ for large enough $k$. Substituting, we get

\begin{align*}
    \frac{2\lfloor \frac{1}{2b} \rfloor - 2k + 2}{2k} &= \frac{2(k + C - 1) - 2k + 2}{2k} = \frac{C}{k} \\
    &= \frac{C(2k - 1)}{k(2k - 1} = \frac{(2k^2 - k) - (k - C)(2k - 1)}{2k^2 - k} \\
    &= 1 - (2k - 1)\frac{k - C}{2k^2 - k} \\
    &= 1 - mb
\end{align*}

Therefore, the calculated value for type $A$'s weight is equal to $a$, and similarly the original value for type $B$'s weight is equal to $b$. All examples of this form must are indeed non-trivial fixed points. \qedhere

To show that these classify all the fixed points of $[1/2_s: a, b, b, \dots]$, we note that $\lfloor \frac{1}{2b} \rfloor$ must be an integer, so we know there exists an integer $C$ such that $\lfloor \frac{1}{2b} \rfloor = k + C - 1$. Substituting this in and solving for $b$, we get $b = \frac{k - C}{2k^2 - k}$. We also know that $C > 0$ because if $C \leq 0$, then $\frac{k - C}{2k^2 - k} \geq \frac{1}{2k - 1}$, so this value of $b$ represents either a trivial case or an impossible one where $(2k - 1)b > 1$. Since there are no non-trivial fixed points in the other case, this classifies all the fixed points if $m = 2k - 1$ for $[1/2_s: a, b, \dots, b]$.
\end{proof}

\begin{proposition}
Let $C$ be a positive integer. If $m = 2k$, then for large enough $k$, $b = \frac{k - C}{2k^2 + k}$ and $a = 1 - 2kb$ gives a non-trivial fixed point of the form $[1/2_s: a, b, \ldots, b]$. Moreover, all non-trivial fixed points are of this form.
\end{proposition}

\begin{proof}
Since $\frac{1}{2b} \not \in \mathbb{Z}$, we know that $\frac{2\lfloor \frac{1}{2b} \rfloor - 2k + 1}{2k + 1} = 1 - 2kb$ is necessary for this value of $b$ to be a fixed point. If we plug in $b = \frac{k - C}{2k^2 + k}$ into $\lfloor \frac{1}{2b} \rfloor$, we get that this equals $k + C$ for large enough $k$. Substituting this in, we get

\begin{align*}
    \frac{2\lfloor \frac{1}{2b} \rfloor - 2k + 1}{2k + 1} &= \frac{2C + 1}{2k + 1} \\ 
    &= \frac{2kC + k}{2k^2 + k} \\ 
    &= \frac{(2k^2 + k) - 2k(k - C)}{2k^2 + k} \\
    &= 1 - 2kb
\end{align*}

Therefore, the calculated value for type $A$'s weight is equal to $a$, and similarly the original value for type $B$'s weight is equal to $b$. All examples of this form are hence non-trivial fixed points. 
\qedhere 
To show that these classify all the fixed points of $[1/2_s: a, b, b, \dots]$, we note that $\lfloor \frac{1}{2b} \rfloor$ must be an integer, so we know there exists an integer $C$ such that $\lfloor \frac{1}{2b} \rfloor = k + C$. Substituting this in and solving for $b$, we get $b = \frac{k - C}{2k^2 + k}$. We also know that $C > 0$ because if $C \leq 0$, then $\frac{k - C}{2k^2 + k} \geq \frac{1}{2k + 1}$, so this value of $b$ represents either a trivial case or an impossible one where $2kb > 1$. Since there are no non-trivial fixed points in the other case, this classifies all the fixed points if $m = 2k$ for $[1/2_s: a, b, b, \dots]$.
\end{proof}

We can now take a look at weighted voting systems in the form of $(a, a, b, \ldots, b)$, which we can analyze in a similar way to the previous parts. Again letting $m$ be the number of players of type $B$, we can then split into two cases.

\begin{enumerate}
    \item $m = 2k$
    
    In this case, there are two possibilities for a winning coalition in which the critical player is in type $A$. Either the other player of type $A$ is in the winning coalition or they are not. We can calculate how many permutations fit into one of the categories and then add them together to get the total number of permutations where a person of type $A$ is the critical player when trying to compute Shapley-Shubik.
    
    \begin{itemize}
        \item Only one player of type $A$ in the winning coalition.
        
        Let $p$ be the number of players of type $B$ in the winning coalition. We know that $pb \leq \frac{1}{2}$ and $pb + a > \frac{1}{2}$. Solving for $p$ and substituting $a = \frac{1 - 2kb}{2}$, we get the inequalities $p \leq \frac{1}{2b}$ and $p > k$. Since $p$ must be an integer, we can make stricter bounds by saying $k + 1 \leq p \leq \lfloor \frac{1}{2b} \rfloor$.
        
        Now, for a given value of $p$, we need to figure out how many permutations work with that value of $p$. Out of the $2k + 2$ people voting, the first $p + 1$ spots are part of the winning coalition in a given permutation, so the other player of type $A$ must be in the remaining $2k - p + 1$ locations. There are $(2k)!$ ways to arrange the rest of the players, so in total, there are $m!(2k - p + 1)$ permutations associated with a given value of $p$. We then need to sum this expression using the bounds from earlier to get the total number of permutations in this case to be $(2k)!\sum_{p = k + 1}^{\lfloor 1/(2b) \rfloor}(2k - p + 1) = \frac{(2k)!}{2}(\lfloor \frac{1}{2b} \rfloor - k)(3k - \lfloor \frac{1}{2b} \rfloor + 1)$.
        
        \item Both players of type $A$ are in the winning coalition.
        
        Again, let $p$ be the number of players of type $B$ in the winning coalition. We know that $pb + a \leq \frac{1}{2}$ and $pb + 2a > \frac{1}{2}$. Solving for $p$ and substituting $a = \frac{1 - 2kb}{2}$, we get the inequalities $p \leq k$ and $p > 2k - \frac{1}{2b}$. If we assume that $\frac{1}{2b} \not \in \mathbb{Z}$, then since $p$ must be an integer, we can write $2k - \lfloor \frac{1}{2b} \rfloor \leq p \leq k$. For each value of $p$, there are $p + 1$ places where we can place the other player of type $A$ and $(2k + 1)!$ ways to place the rest of the players, so there are $(2k)!\sum_{p = 2k - \lfloor \frac{1}{2b} \rfloor}^k (p + 1) = \frac{1}{2}\left(\lfloor \frac{1}{2b} \rfloor - k + 1\right)\left(3k - \lfloor \frac{1}{2b} \rfloor + 2\right)$.
     \end{itemize}
    
    Combining these two cases, along with the fact that there are $(2k + 2)!$ total permutations, we get that the power of a player of type $A$ must be $$\frac{\frac{1}{2}\left(\lfloor \frac{1}{2b} \rfloor - k\right)\left(3k - \lfloor \frac{1}{2b} \rfloor + 1\right) + \frac{1}{2}\left(\lfloor \frac{1}{2b} \rfloor - k + 1\right)\left(3k - \lfloor \frac{1}{2b} \rfloor + 2\right)}{(2k + 1)(2k + 2)},$$ and a value of $b$ is a fixed point if and only if this expression equals $\frac{1 - 2kb}{2}$. Using this equation, we can generate the following solutions for $b$.

    \begin{center}
    \begin{tabular}{c|l}
    $k = 1$ & $b = 1/3$ \\ \hline
    $k = 2$ & $b = 2/15, b = 1/5$ \\ \hline
    $k = 3$ & $b = 3/28, b = 1/7$ \\ \hline 
    $k = 4$ & $b = 13/180, b = 4/45, b = 1/9$ \\ \hline 
    $k = 5$ & $b = 7/110, b = 5/66, b = 1/11$ 
    \end{tabular}
    \end{center}
    
    We can verify that $b = \frac{1}{2k + 1}$ and $b = \frac{k}{(k + 1)(2k + 1)}$ are classes of primitive non-trivial fixed points for large enough $k$ by showing that both sides equal each other when we substitute the values of $b$ in.
    
    \item $m = 2k + 1$
    
    We again split into different cases depending on the number of players of type $A$ in the winning coalition.
    
    \begin{itemize}
    \item Only one player of type $A$ in the winning coalition.
    
    Let $p$ be the number of players of type $B$ in the winning coalition. We know that $pb \leq \frac{1}{2}$ and $pb + a > \frac{1}{2}$. Solving for $p$ and substituting $a = \frac{1 - (2k  +1)b}{2}$, we get the inequalities $p \leq \frac{1}{2b}$ and $p > k$. Since $p$ must be an integer, we can make stricter bounds by saying $k + 1 \leq p \leq \lfloor \frac{1}{2b} \rfloor$.
    
    Now, for a given value of $p$, we need to figure out how many permutations work with that value of $p$. Out of the $2k + 3$ people voting, the first $p + 1$ spots are part of the winning coalition in a given permutation, so the other player of type $A$ must be in the remaining $2k - p + 2$ locations. There are $(2k + 1)!$ ways to arrange the rest of the players, so in total, there are $(2k + 1)!(2k - p + 2)$ permutations associated with a given value of $p$. We then need to sum this expression using the bounds from earlier to get the total number of permutations in this case to be $(2k + 1)!\sum_{p = k + 1}^{\lfloor 1/(2b) \rfloor}(2k - p + 2) = \frac{(2k + 1)!}{2}(\lfloor \frac{1}{2b} \rfloor - k)(3k - \lfloor \frac{1}{2b} \rfloor + 3)$.
    
    \item Both players of type $A$ are in the winning coalition.
    
    Again, let $p$ be the number of players of type $B$ in the winning coalition. We know that $pb + a \leq \frac{1}{2}$ and $pb + 2a > \frac{1}{2}$. Solving for $p$ and substituting $a = \frac{1 - (2k + 1)b}{2}$, we get the inequalities $p \leq k$ and $p > 2k + 1 - \frac{1}{2b}$. If we assume that $\frac{1}{2b} \not \in \mathbb{Z}$, then since $p$ must be an integer, we can write $2k + 1 - \lfloor \frac{1}{2b} \rfloor \leq p \leq k$. For each value of $p$, there are $p + 1$ places where we can place the other player of type $A$ and $(2k + 1)!$ ways to place the rest of the players, so there are $(2k + 1)!\sum_{p = 2k + 1 - \lfloor \frac{1}{2b} \rfloor}^k (p + 1) = \frac{1}{2}\left(\lfloor \frac{1}{2b} \rfloor - k\right)\left(3k - \lfloor \frac{1}{2b} \rfloor + 3\right)$.
\end{itemize}
Combining these two cases, along with the fact that there are $(2k + 3)!$ total permutations, we get that the power of a player of type $A$ must be $$\frac{\left(\lfloor \frac{1}{2b} \rfloor - k\right)\left(3k - \lfloor \frac{1}{2b} \rfloor + 3\right)}{(2k + 1)(2k + 2)},$$ and a value of $b$ is a fixed point if and only if this expression equals $\frac{1 - (2k + 1)b}{2}$. Using this equation, we can generate the following solutions for $b$.
\begin{center}
\begin{tabular}{c|l}
$k = 1$ & $b = 2/15$ \\ \hline
$k = 2$ & $b = 3/35, b = 11/105$ \\ \hline
$k = 3$ & $b = 4/63, b = 11/126$ \\ \hline 
$k = 4$ & $b = 5/99, b = 31/495, b = 37/495$ 
\end{tabular}
\end{center}

\noindent And we can verify that $b = \frac{k + 1}{4(k + 1)^2 - 1}$ is a primitive fixed points for large enough $k$.
\end{enumerate}

\subsection{Both Banzhaf and Shapley-Shubik}
To find primitive nontrivial fixed points for both power indices, we can test Shapley-Shubik fixed points to see if they are still fixed points for Banzhaf. Consider the Shapley-Shubik fixed points in $\textbf{Proposition 3.1}$. To calculate Banzhaf power indices for this class of fixed points, we see that each winning coalition either contains the player of type $A$ or it does not.
\begin{itemize}
    \item Winning coalition does not contain player of type A: The winning coalition must contain at least $p$ players of type $B$, where $p(\frac{k-c}{2k^2-k}) > \frac{1}{2}$, or $p > \frac{2k^2-k}{2k-2c} = k + c - \frac{1}{2} + \frac{2c^2-c}{2k-2c}$. Since $p$ is an integer, we have $p \geq k + c$ for large enough $k$. In order for the coalition to have a critical player, we must have $p = k + c$, so each player of type $B$ is a critical player $\binom{2k-2}{k+c-1}$ times since we pick the other $p - 1$ players from the remaining $2k-2$ players of Type $B$. 
    \item Winning coalition contains player of type A:  The winning coalition must contain $q$ players of type $B$, where $q(\frac{k-c}{2k^2-k}) < \frac{1}{2}$ and $\frac{1}{k} + q(\frac{k-c}{2k^2-k}) > \frac{1}{2}$, which simplifies to $\frac{(k-2c)(2k-1)}{2k-2c} < q < \frac{2k^2-k}{2k-2c}$. The lower and upper bounds are equal to $k - c - \frac{1}{2} + \frac{c-2c^2}{2k-2c}$ and $k + c - \frac{1}{2} + \frac{2c^2-c}{2k-2c}$, respectively, so since $q$ is an integer, we have $k - c \leq q \leq k + c - 1$ for large enough $k$. Then the player of type $A$ is a critical player $\binom{2k-1}{k-c} + \binom{2k-1}{k-c+1} + \dots + \binom{2k-1}{k+c-1}$ times, where we choose $q$ players of type $B$ from the total $2k-1$. Furthermore, players of type $B$ are also critical players when $q = k-c$, so they are critical players $\binom{2k-2}{k-c-1}$ times. 
\end{itemize}
Thus the Banzhaf index for a player of type $B$ is given by 
\[\dfrac{\binom{2k-2}{k+c-1} + \binom{2k-2}{k-c-1}}{(2k-1)\left( \binom{2k-2}{k+c-1} + \binom{2k-2}{k-c-1} \right) + \binom{2k-1}{k-c} + \binom{2k-1}{k-c+1} + \dots + \binom{2k-1}{k+c-1}}.
\]
We want this expression to equal $\frac{k-c}{2k^2-k}$, which we found to only hold true when $c=1$.
\newline
\\
Now consider the Shapley-Shubik fixed points in $\textbf{Proposition 3.2}$. To calculate Banzhaf power indices for this class of fixed points, we see that each winning coalition either contains the player of type $A$ or it does not.
\begin{itemize}
    \item Winning coalition does not contain player of type A: The winning coalition must contain at least $p$ players of type $B$, where $p(\frac{k-c}{2k^2+k}) > \frac{1}{2}$, or $p > \frac{2k^2+k}{2k-2c} = k + c + \frac{1}{2} + \frac{2c^2+c}{2k-2c}$. Since $p$ is an integer, we have $p \geq k + c +1$ for large enough $k$. In order for the coalition to have a critical player, we must have $p = k + c+1$, so each player of type $B$ is a critical player $\binom{2k-1}{k+c}$ times since we pick the other $p - 1$ players from the remaining $2k-1$ players. 
    \item Winning coalition contains player of type A:  The winning coalition must contain $q$ players of type $B$, where $q(\frac{k-c}{2k^2+k}) < \frac{1}{2}$ and $\frac{1}{k} + q(\frac{k-c}{2k^2+k}) > \frac{1}{2}$, which simplifies to $\frac{(k-2c)(2k+1)}{2k-2c} < q < \frac{2k^2+k}{2k-2c}$. The lower and upper bounds are equal to $k - c + \frac{1}{2} - \frac{c+2c^2}{2k-2c}$ and $k + c + \frac{1}{2} + \frac{2c^2+c}{2k-2c}$, respectively, so since $q$ is an integer, we have $k - c +1\leq q \leq k + c $ for large enough $k$. Then the player of type $A$ is a critical player $\binom{2k}{k-c+1} + \binom{2k}{k-c+2} + \dots + \binom{2k}{k+c}$, where we choose $q$ players of type $B$ from the total $2k-1$. Furthermore, players of type $B$ are also critical players when $q = k-c+1$, so they are critical players $\binom{2k-1}{k-c}$ times. 
\end{itemize}
Thus the Banzhaf index for a player of type $B$ is given by 
\[ \dfrac{\binom{2k-1}{k+c} + \binom{2k-1}{k-c}}{(2k)\left( \binom{2k-1}{k+c} + \binom{2k-1}{k-c} \right) + \binom{2k}{k-c+1} + \binom{2k}{k-c+2} + \dots + \binom{2k}{k+c}}.
\]
We want this expression to equal $\frac{k-c}{2k^2+k}$, but we found no integer solutions for $c$ for which this is true.

\section{Future Research}
To expand the class of abundant numbers for which Proposition 2.1 is true, it would be helpful to prove that if the Shapley-Shubik and Banzhaf power indices differ in the divisor voting system of an abundant number $n$, then they also differ in the divisor voting system of $2^{k}p$ for a prime $p$ and various values of $k$. It would also be useful to finish our proof that the players in the divisor voting system of $pn$ have equal Banzhaf indices as well as equal Shapley-Shubik indices as those in $mn$.
By proving these statements, it would then be possible to prove Question $1.2.$ by inducting on the number of distinct prime divisors of an abundant number $n$.
In further research concerning fixed points, we could try strengthening the algebraic approach for fixed points to find more classes (e.g. with three types of players).

\section{Acknowledgements}
    
We would like to thank Miroslav Marinov for his guidance, mentorship, and feedback on this research paper. We thank Joshua Zelinsky for proposing this research project and for his helpful suggestions. We would like to thank the PROMYS program, especially Professor Fried and Steve Huang for the opportunity to work on this project, and the Clay Mathematical Institute for their support of the returning student program.

\section{Appendix}
Here is the code that we used to compute the power indices. 
\begin{lstlisting}[language=Python]
import itertools

def sumAll(list):
    sum = 0
    for num in list:
        sum += num
    return sum

def contains(list, key):
    for element in list:
        if element == key:
            return true
    return false

def computeShapleyShubikIndicesDivisors(n):
    divisors = []
    for i in range(1, ceil(sqrt(n))):
        if n % i == 0:
            divisors.append(i)
            if i * i < n:
                divisors.append(n // i)
    div = dict([(i, i) for i in sorted(divisors)])
    return computeShapleyShubikIndicesWeights(div)

def computeShapleyShubikIndicesWeights(weights):
    s = sum([weights[key] for key in weights])
    minVotes = ceil((s + 1) / 2)
    counts = dict([(key, 0) for key in weights])
    return computeShapleyShubikIndices(minVotes, weights, counts)

def computeShapleyShubikIndices(minVotes, weights, counts):
    # Count # times each player is a pivotal player
    sum=0
    for tuple in itertools.permutations(weights.keys(),len(weights)):
        sum = 0;
        for player in tuple:
            sum += weights[player]
            if (sum >= minVotes):
                counts[player] = counts.get(player)+1
                break

    print("SS numerator: ", counts)

    # Normalize by the total # times a player is a pivotal player - len(weights)!
    print("SS Denominator (n!): ", factorial(len(weights)))
    totalSum = factorial(len(weights))
    normalizedCounts = []
    for value in counts.values():
        normalizedCounts.append(value/totalSum)

    return normalizedCounts
    
def sumAll(list):
    sum = 0
    for num in list:
        sum += num
    return sum

def contains(list, key):
    for element in list:
        if element == key:
            return true
    return false

def isCriticalPlayer(player, coalition, coalitionWeights, weights, minVotes):
    if contains(coalition, player) and sumAll(coalitionWeights)-weights[player] < minVotes:
        return true
    return false

def computeBanzhafIndicesDivisors(n):
    divisors = []
    for i in range(1, ceil(sqrt(n))):
        if n % i == 0:
            divisors.append(i)
            if i * i < n:
                divisors.append(n // i)
    div = dict([(i, i) for i in sorted(divisors)])
    return computeBanzhafIndicesWeights(div)

def computeBanzhafIndicesWeights(weights):
    s = sum([weights[key] for key in weights])
    minVotes = ceil((s + 1) / 2)
    counts = dict([(key, 0) for key in weights])
    return computeBanzhafIndices(minVotes, weights, counts)

def computeBanzhafIndices(minVotes, weights, counts):
    # Compute all winning coalitions
    winningCoalitions = []
    for subset in Subsets(weights.keys(), submultiset=True):
        listOfWeights = []
        for player in subset:
            listOfWeights.append(weights[player])
        if len(listOfWeights)== 0 or sumAll(listOfWeights) < minVotes:
            continue
        winningCoalitions.append(subset)

    # Count # times each player is a critical player
    for player in weights.keys():
        count = 0
        for coalition in winningCoalitions:
            coalitionWeights = []
            for player2 in coalition:
                coalitionWeights.append(weights[player2])
            if isCriticalPlayer(player, coalition, coalitionWeights, weights, minVotes):
                count += 1
        counts[player] = count
    print("Original Banzhaf Indices Counts: ", counts)

    # Normalize by the total # times a player is a critical player
    print("Total Sum of Counts: ", sumAll(counts.values()))
    totalSum = sumAll(counts.values())
    normalizedCounts = [counts[key] / totalSum for key in counts]

    return normalizedCounts

# Test for BI
print(computeBanzhafIndicesDivisors(18))
# Test for SSI
print(computeShapleyShubikIndicesDivisors(18))
\end{lstlisting}

\end{document}